\documentclass[11pt]{amsart}

\usepackage{afterpage}
\usepackage{amsmath,amsfonts,amssymb,amsthm}
\usepackage{caption}
\usepackage[inline]{enumitem}
\usepackage{graphicx}
\usepackage{hyperref}
\usepackage[section]{placeins}
\usepackage{tabularx}
\usepackage[obeyDraft, textwidth=3.7cm]{todonotes}

\DeclareSymbolFont{bbold}{U}{bbold}{m}{n}
\DeclareSymbolFontAlphabet{\mathbbold}{bbold}

\DeclareMathOperator{\Aut}{{\bf Aut}}
\DeclareMathOperator{\End}{{\bf End}}
\DeclareMathOperator{\ncl}{{\bf ncl}}
\DeclareMathOperator{\AK}{{AK}}
\DeclareMathOperator{\AC}{{AC}}
\DeclareMathOperator{\Loop}{{Loop}}
\DeclareMathOperator{\Fold}{{Fold}}
\DeclareMathOperator{\Shift}{{Shift}}
\DeclareMathOperator{\Id}{{Id}}
\DeclareMathOperator{\NormForm}{{NF}}

\newcommand{\gpr}[2]{{\left\langle #1 \mid #2 \right\rangle}}
\newcommand{\rb}[1]{{\left( #1 \right)}}
\newcommand{\Set}[2]{\left\{\, #1 \;\middle|\; #2 \,\right\}}

\newcommand{\NF}[1]{{\NormForm(#1)}}

\def\MN{{\mathbb{N}}}

\def\MZ{{\mathbb{Z}}}
\def\MR{{\mathbb{R}}}

\def\ovu{{\overline{u}}}
\def\ovv{{\overline{v}}}

\def\One{{\mathbbold{1}}}

\newcommand{\EP}{\mathbb{R}^2}
\newcommand{\CB}{{\mathcal B}}
\newcommand{\CC}{{\mathcal C}}

\newtheorem{thm}{Theorem}[section]
\newtheorem{cor}[thm]{Corollary}
\newtheorem{lem}[thm]{Lemma}
\newtheorem{prop}[thm]{Proposition}
\newtheorem*{theorem*}{Theorem}
\newtheorem*{lemma*}{Lemma}
\newtheorem*{proposition*}{Proposition}
\newtheorem*{conj*}{Conjecture}

\author[D. Panteleev]{Dmitry Panteleev}
\author[A. Ushakov]{Alexander Ushakov}
\address{Department of Mathematics, Stevens Institute of Technology, Hoboken, NJ, USA}
\email{dpantel1,aushakov@stevens.edu}
\thanks{The second author has been partially supported by NSF grant DMS-1318716}

\title[On Andrews--Curtis conjecture]{Conjugacy search problem and the Andrews--Curtis conjecture}
\date{\today}

\begin{document}
\maketitle

\begin{abstract}
We develop new computational methods for studying
potential counterexamples to the Andrews--Curtis conjecture,
in particular, Akbulut--Kurby examples AK(n).
We devise a number of algorithms in an attempt to disprove
the most interesting counterexample AK(3).
To improve metric properties of the search space
(which is a set of balanced presentations of $\One$)
we introduce a new transformation (called an ACM-move here) that generalizes
the original Andrews-Curtis transformations and discuss details of a practical implementation.
To reduce growth of the search space we introduce a strong equivalence
relation on balanced presentations and study the space modulo
automorphisms of the underlying free group.
Finally, we prove that automorphism-moves can be applied to
AK(n)-presentations.
Unfortunately, despite a lot of effort we were unable to trivialize
any of AK(n)-presentations, for $n>2$.

\end{abstract}

\section{Introduction}

The Andrews--Curtis conjecture (AC-conjecture, or ACC)
is a long-standing open problem in low-dimensional
topology and combinatorial group theory. It was proposed
by Andrews and Curtis in \cite{Andrews-Curtis:1965} while categorizing possible
counterexamples to the Poincar\'e conjecture. Later, Wright in
\cite{wright1975group} formulated an equivalent conjecture about $3$-deformations
of 2-CW-complexes associated with all finitely presented groups, thus showing that
Zeeman conjecture \cite{zeeman1963dunce} implies AC-conjecture.
It is known that Zemman conjecture also implies the Poincar\'e conjecture and
is implied by Poincar\'e in some cases, e.g. \cite{gillman1983zeeman}.
Despite recent progress and solution to Poincar\'e conjecture,
validity of the AC-conjecture remains open.

Although most of motivating examples come from topology,
the conjecture is usually formulated in the language of
combinatorial group theory, as a question of equivalence of
presentations of the trivial group.
In this paper we use the language of combinatorial group theory,
omitting any topological aspects of the problem.

\subsection{Balanced presentations of the trivial group}

Let $X=\{x_1,\ldots,x_n\}$, $F=F(X)$ be the free group on $X$, and $R$
a finite subset of $F$. The normal closure of $R$, denoted by $\ncl(R)$,
is the smallest normal subgroup in $F(X)$ containing $R$.
A pair $(X;R)$ defines a quotient group $F/\ncl(R)$,
denoted by $\gpr{X}{R}$, and is called a \emph{presentation}
of $\gpr{X}{R}$. The sum $\sum_{r\in R} |r|$ is called the
\emph{total length} of the presentation $(X;R)$ and is denoted by $L(R)$.
We say that $R \subseteq F(X)$ is \emph{symmetrized}
if $R$ contains only cyclically reduced words
and is closed under taking inverses and cyclic permutations.
Denote by $R^{\star}$ the minimal symmetrized set containing $R$ (with all the words cyclically reduced). A presentation $(X;R)$ is \emph{symmetrized} if $R = R^{\star}$. A finite presentation can be efficiently symmetrized and symmetrization does not change the computational properties of the fundamental problems.

We say that a group presentation $(X;R)$
is \emph{balanced} if $|X|=|R|$.
Some group presentations define the trivial group $\One$.
The ``most trivial'' presentation of $\One$ on generators $\{x_1,\ldots,x_n\}$
is, of course,
$(x_1,\ldots,x_n\  ;\ x_1,\ldots,x_n)$ called the
\emph{canonical presentation} of $\One$ on $\{x_1,\ldots,x_n\}$.
Define a set $\CB_n \subseteq F_n^n$ of balanced relator-tuples
of the trivial group:
$$
\CB_n = \Set{(r_1,\ldots,r_n)}{\ncl(r_1,\ldots,r_n)=F_n}.
$$
We use vector notation for tuples in $F_n^n$.
The problem of deciding if $(X;R)$ defines
the trivial group is undecidable (see \cite{Novikov:1955,Boone:1958}).
It is an open problem if the same is true for balanced presentations
(see Magnus' problem \cite[Problem 1.12]{Kourovka}).

\subsection{Transformations of group presentations}

There are several types of transformations
that for a general group presentation $(x_1,\ldots,x_n;r_1,\ldots,r_n)$
produce a new presentation $(x_1,\ldots,x_n;r_1',\ldots,r_n')$
on the same set of generators of the same group.
The Andrews-Curtis transformations $AC_1,AC_2,AC_3$
(or simply AC-moves) are of that type:
\begin{enumerate}[leftmargin=1in]
\item[($AC_1$)] $r_i \to r_i r_j$ for $i\neq j$,
\item[($AC_2$)] $r_i \to r_i^{-1}$,
\item[($AC_3$)] $r_i \to w^{-1} r_i w$ for some $w \in F_n$.
\end{enumerate}
The transformations $AC_1,AC_2$ can be recognized as Nielsen transformations
of the tuple $(r_1,\ldots,r_n)$ and the transformation $AC_3$
is a conjugation of any element in a tuple.
Since the AC-moves are invertible,
we can say that $\ovu$ and $\ovv$ are \textbf{AC-equivalent}
(and write $\ovu\sim_{AC}\ovv$)
if there exists a sequence of AC-moves transforming $\ovu$ into $\ovv$.

More generally, a transformation (named here an ACM-move) that
replaces a single element $u_i$
in $\ovu$ with an element $u_i'$ satisfying:
$$
u_i^{\pm 1} \sim u_i' \mbox{ in } \gpr{x_1,\ldots,x_n}{u_1,\ldots,u_{i-1},u_{i+1},\ldots,u_n}
$$
produces an isomorphic presentation.
It is easy to see that AC-moves are particular types of
the ACM-move.
Also, the ACM-move can be recognized as a slightly generalized $M$-transformation of
\cite{Burns-Macedonska:1992}.
It is easy to see that $\ovu$ can be transformed
to $\ovv$ by AC-moves if and only if the same can done by ACM-moves.
Therefore, to check AC-equivalence one can use ACM-moves.

Yet another transformation of a group presentation that does not change the group is
an \emph{automorphism move},
which is an application of $\varphi\in\Aut(F_n)$ to every component of $\ovu$.
It is not known if the system of AC-moves with automorphism-moves
is equivalent to the system of AC-moves.
More on automorphism-moves in Section \ref{se:automorphisms}.

\subsection{The conjecture}
\label{se:conjecture}

Denote by $\CC$ the set of all tuples that can be obtained from the
canonical tuple by a sequence of AC-moves. More generally, for
$\ovu\in\CB_n$ denote by $\CC_\ovu$ the set of tuples in $\CB_n$
$\AC$-equivalent to $\ovu$.

The \emph{Andrews-Curtis conjecture}
\cite{Andrews-Curtis:1966} states that $\CC=\CB_n$, i.e.,
every balanced presentation of the trivial group
can be converted to the canonical presentation
by a sequence of AC-moves.

Despite nearly $50$ years of research the conjecture is still open.
It is widely believed that the Andrews–Curtis conjecture is false
with most theoretic works attempting to disprove it.
A common approach is to fix some group $G$,
a homomorphism $\varphi:F_n\to G$,
and investigate if for any $\ovu\in\CB_n$
there exists a sequence of AC-moves taking $\varphi(\ovu)$
into $(\varphi(x_1), \ldots, \varphi(x_n))$.
Clearly, if the answer is negative for some choice of $G$ and $\varphi$,
then the original conjecture does not hold.
Several classes of groups were investigated that way,
e.g., solvable groups \cite{Miasnikov:1984},
finite groups \cite{Borovik-Lubotzky-Myasnikov:2005},
the Grigorchuk group \cite{Myropolska:2013},
but the (negative) answer is not found.

\subsection{Potential counterexamples}

A big obstacle towards the solution of the problem is that
there is no algorithm to test if a particular
balanced presentation of the trivial group satisfies the conjecture,
or not.
There is a number of particular balanced presentations
that are not known to satisfy the conjecture.
\begin{itemize}
\item
Akbulut--Kurby examples: $\AK(n)=\gpr{x,y}{xyx=yxy,\ x^{n+1}=y^n}$ for $n\ge 3$.
\item
Miller--Schupp examples:
$\gpr{x,y}{x^{-1} y^2 x = y^3,\ x=w},$
where $w$  has exponent sum $0$ on $x$.
\item
B.~H.~Neumann example $\gpr{x,y,z}{z^{-1} y z=y^2,\ x^{-1}zx=z^2,\ y^{-1}xy=x^2}$.
\end{itemize}
These examples are referred to as \emph{potential counterexamples} to ACC.
More examples of balanced presentations of $\One$ (known to be AC-equivalent
to the canonical presentation) can be found
in \cite{Ivanov:2006, lishak2015balanced, bridson2015complexity}.
It was shown in \cite{Miasnikov99geneticalgorithms}
by means of a computer experiment
that there are no counterexamples
of total length $12$ or less. Later it was shown
in \cite{HavasRamsay03} that every balanced presentation
of total length $13$ is either AC-equivalent to the canonical
presentation or to
$$
\AK(3) = \langle x, y \mid x^3 = y^4, xyx = yxy\rangle,
$$
which makes $\AK(3)$ the shortest potential counterexample.





\subsection{Computational approach to disproving a counterexample}
\label{se:computational_AC}

To check if a given tuple $\ovu$ is $\AC$-equivalent to the canonical presentation,
one can enumerate equivalent presentations (by applying AC-moves)
until the canonical presentation is found
(see \cite{Miasnikov99geneticalgorithms,HavasRamsay03,bowman2006}).
There are several general computational problems associated with that approach
that we would like to mention here:
\begin{itemize}
\item
$\CC_\ovu$ is infinite and there is no terminating condition
which allows an enumeration procedure to stop with a negative answer.
The enumeration procedure can only terminate with a positive answer
when it finds the canonical presentation.
\item
Lengths of tuples are unbounded.
\item
$\CC_\ovu$ has exponential growth.
\end{itemize}
To alleviate some of the problems
one can bound the lengths of the words in tuples by some constant $L$ and
do not process a tuple $\ovv$ which is $\AC$-equivalent to the given $\ovu$
if $\ovv$ contains a (cyclic) word of length greater than $L$.
This approach (used in \cite{HavasRamsay03,bowman2006}) allows to use fixed memory slots for words
and makes the search space finite.
Also, it is a good heuristic to process shorter tuples first.

In this paper we consider the case $n=2$ only.
We use compact memory representation for balanced pairs $(u,v)\in\CB_2$.
We represent each letter by a $2$-bit number, thus packing $32$ letters into
a $64$-bit machine-word. This approach saves memory and allows to
implement operations such as a cyclic shift in just a few processor
instructions, compared to a usual approach which includes
several memory writes.

\subsection{Our work}

In this paper we develop new efficient techniques to enhance algorithmic search in $\CC$
(or $\CC_\ovu$).
Our work is similar to previous computational investigations of ACC, but goes much further.
The presentation $\AK(3)$ is the main object of study and
most of the algorithms are tested on $\AK(3)$. Our big goal was to
prove that $\AK(3)$ is not a counterexample, i.e., it satisfies ACC.
Unfortunately, we were not able to achieve our goal.
Below we list the key features of our work.

\begin{itemize}
\item
In Section \ref{se:newmove} we show that ACM-moves
can be used in practice. Notice that for $n=2$
that requires enumerating (short!) conjugates in a one relator group
for a given element.
The later problem does not have an efficient solution as of now.
It is not even known if the conjugacy problem is decidable
or not in one relator groups.
Based on techniques described in \cite{ushakov2011non} we design
a heuristic procedure enumerating short conjugates
and discuss the details of implementation.
\item
We prove in Section \ref{se:automorphisms} that automorphism-moves
can be used with the AC-moves
for Akbulut-Kurby presentations $\AK(n)$, regardless of whether the conjecture holds
for $\AK(n)$.
\item
In Section \ref{se:equivalence_rel} we introduce an equivalence relation
$\sim$ on pairs in $\CB_2$ and define normal forms for the equivalence classes.
We show that in practice equivalence of two pairs can be checked and normal forms computed.
That allows us to work with the quotient space $\CB_2/\sim$ which
elements are (infinite) equivalence classes.
Working in $\CB_2/\sim$ we work with large blocks of elements from $\CB_2$.
Thus, we can say that
the space $\CB_2/\sim$ is much smaller than $\CB_2$,
even though both sets are infinite countable.
\item
In Section \ref{se:org_groups} we use heuristics to investigate if
certain properties of one-relator groups described
in \cite{bridson2015complexity} and \cite{lishak2015balanced}
could be the reason of our unsuccessful search of trivialization for $\AK(3)$.
\item
In Section \ref{se:results} we present results of our experiments.
\end{itemize}

\section{ACM-move}
\label{se:newmove}

In this section we describe our implementation of the ACM-move, i.e.,
an algorithm which for a given pair $u,v\in F=F(x,y)$
constructs a subset of the set:
$$
U = U(u,v)= \Set{u'\in F(x,y) }{u'\sim u \mbox{ in } \gpr{x,y}{v} \mbox{ and } |u'|\le L},
$$
where $L\in\MN$  is a fixed parameter value.
Ideally, the algorithm should construct the whole set $U(u,v)$.
The algorithm is based on \emph{weighted $X$-digraphs}.

\subsection{Weighted $X$-digraphs}

Formally, a weighted $X$-digraph is a tuple $(V,E,\mu,\gamma)$
where $(V,E)$ defines a directed graph, $\mu:E\to X^\pm$ is the
\emph{labeling function}, and $\gamma:E\to\MZ$ is the \emph{weight function}.

We often use the following notation $a \stackrel{x,k}{\to} b$
for the edge with origin $a$, terminus $b$, label $x$,
and weight $k$.
We say that an edge $b \stackrel{x^{-1},-k}{\to} a$ is the inverse to
$e=a \stackrel{x,k}{\to} b$ and denote it by $e^{-1}$.
We say that a weighted $X$-digraph $\Gamma$ is:
\begin{itemize}
\item
\emph{folded}, if for every $a\in V$ and $x\in X^\pm$ there exists
at most one edge with the origin $a$ labeled with $x$;
\item
\emph{inverse}, if with every edge $e$ the graph $\Gamma$ contains $e^{-1}$;
\item
\emph{rooted}, if $\Gamma$ comes with a designed vertex called the \emph{root}.
\end{itemize}
A \emph{path} $p$ in $\Gamma$ is a sequence of adjacent edges $e_1\ldots e_n$,
its label is $\mu(p)=\mu(e_1)\ldots\mu(e_k)$ and the weight $\gamma(p)=\gamma(e_1)+\ldots+\gamma(e_k)$.
A \emph{circuit} is a path with the same origin and terminus.

An inverse weighted labeled digraph $\Gamma$ with a root $v_0$ and a number
$N \in \mathbb{N} \cup \{\infty\}$ is called a {\em pseudo conjugacy graph} for $u$
in $G = \langle x, y\mid v\rangle$ if the following conditions are satisfied:
\begin{itemize}
\item[(CG1)]
$\mu(l) \sim_v u^{\gamma(l)}$ for any circuit $l$.
\item[(CG2)]
$N = \infty$ or $u^N = 1$ in $G$.
\end{itemize}
The simplest nontrivial example of a pseudo conjugacy graph $\Loop(u)$ for $u$
in $\gpr{x,y}{v}$ is shown in Fig.~\ref{fig:loop1}.
\begin{figure}[h]
	\centerline{\includegraphics[width=1.5in]{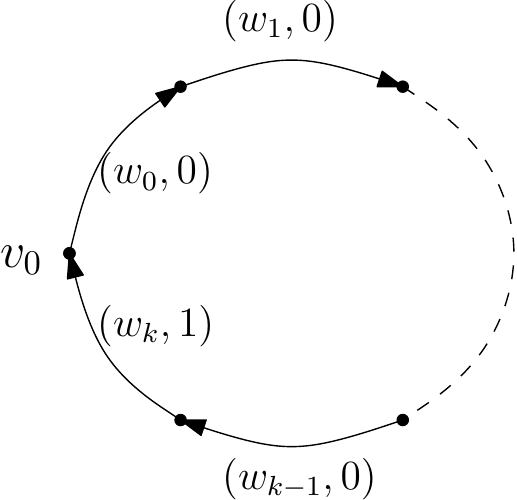}}
	\caption{A pseudo conjugacy graph $\Loop(u)$.}
    \label{fig:loop1}
\end{figure}

To implement ACM-move we generate a sufficiently large pseudo-conjugacy
graph $\Gamma$ for $u$ and then ``harvest'' circuits of weight $1$.
A large pseudo-conjugacy graph $\Gamma$ can be generated
starting with $\Loop(u)$ and applying \emph{$R$-completion}
procedure $D$ times.
$R$-completion is a variation of coset enumeration
first described in \cite{Ushakov:thesis,ushakov2011non}
and reviewed for more precise complexity bounds below.
Harvest is shortly discussed in Section \ref{se:harvest}.

\subsection{Operations on weighted $X$-digraphs}
\label{se:graph_operations}

Here we shortly describe several operations with graphs used later in the sequel.
\subsubsection{Vertex-identification}
Given a folded $X$-digraph $\Gamma$ and distinct $v_1,v_2\in V(\Gamma)$ define
a graph $\Id(\Gamma,v_1,v_2)$ obtained from $\Gamma$ as follows:
\begin{itemize}
\item
add a new vertex $v$ to $V(\Gamma)$;
\item
for each edge $v_i\stackrel{x,a}{\to} u$ add an edge $v\stackrel{x,a}{\to} u$;
\item
for each edge $u\stackrel{x,a}{\to} v_i$ add an edge $u\stackrel{x,a}{\to} v$;
\item
remove $v_1$ and $v_2$.
\end{itemize}
In general, the graph $\Id(\Gamma,v_1,v_2)$ is not folded.

\subsubsection{Weight-shift}
For $v\in V(\Gamma)$ and $\delta\in\MZ$ define a graph
$\Shift(v, \delta)$ obtained from $\Gamma$ by changing weights of edges incident
to $v$ as follows:
\begin{itemize}
\item
the weight $a$ of $e = v \stackrel{x,a}{\to} u$ (where $v\ne u$) is increased by $\delta$;
\item
the weight $a$ of $e = u \stackrel{x,a}{\to} v$ (where $v\ne u$) is decreased by $\delta$.
\end{itemize}
It is easy to see that a weight-shift preserves the weights of circuits
in $\Gamma$ and, hence, preserves the property to be a pseudo-conjugacy graph.
The number of arithmetic operations required for this operation is clearly
bounded by the number of edges $E_v$ incident on $v$.

\subsubsection{Folding}\label{se:folding}
If $\Gamma$ is not folded, then it contains two edges
$e_1=u \stackrel{x,a}{\to} v_1$ and $e_2=u \stackrel{x,b}{\to} v_2$.
Consider several cases.
\begin{itemize}
\item
If $v_1=v_2$ and $a\equiv b \mod N$, then we remove one of the edges.
\item
If $v_1=v_2$ and $a\not\equiv b \mod N$, then we replace the modulus $N$ with
the number $\gcd(N,b-a)$ and remove one of the edges.
\item
If $v_1\ne v_2$, then:
\begin{itemize}
\item
apply $\Shift(v_1,a-b)$  (or $\Shift(v_2,b-a)$) to achieve $\gamma(e_1)=\gamma(e_2)$;
\item
identify $v_1$ and $v_2$;
\item
remove $e_2$.
\end{itemize}
\end{itemize}
It is straightforward to check that the described operation produces
a pseudo-conjugacy graph from a pseudo-conjugacy graph $\Gamma$
(see \cite{Ushakov:thesis} for more detail).
Since folding $e_1$ with $e_2$ decreases the number of edges,
a sequence of folds eventually stops with a folded graph.
The final result of folding is unique up to shifts of weights.
Denote it by $\Fold(\Gamma)$.

\subsubsection{$R$-completion}
Recall that a (finite) set $R\subset F(X)$ is called \emph{symmetrized}
if $R$ with every $r\in R$ contains all cyclic permutations of $r$.
\emph{To complete} a given weighted $X$-digraph $\Gamma$ with (symmetrized)
relators $R\subset F(X)$ means to
add a circuit at $v$ labeled with $r$ to $\Gamma$ of weight $0$
for every $v\in V$ and $r\in R$.
It is easy to check that if $\Gamma$ is a pseudo-conjugacy graph
then the result is a pseudo-conjugacy graph as well.
It clearly requires linear time (in $|\Gamma|$) to $R$-complete a given graph $\Gamma$.
In general, the result is not folded.

\subsection{Complexity of weighted $X$-digraph folding}

Let $\Gamma=(V,E,\mu,\gamma)$ be a weighted $X$-digraph.
It follows from the description of folding that $\Fold(\Gamma)$ as an
$X$-digraph (i.e., if we forget about the weight function) is the same
as the result of folding of the $X$-digraph $(V,E,\mu)$.
The only difference between weighted $X$-digraph folding and $X$-digraph folding
is weight-processing (applications of $\Shift(v_1,a-b)$ or $\Shift(v_2,b-a)$
in Section \ref{se:folding}).
The idea is to modify the procedure and take weights into account.

\subsubsection{$X$-digraph folding}
\label{se:digraph_folding}

Recall that $X$-digraph folding can be done in nearly linear time
$O(|V| \log^\ast(|V|))$, where $\log^\ast$ is inverse-Ackermann function
(see \cite{touikan06}).
In some sense, folding of an $X$-digraph $\Gamma$ describes the following equivalence
relation $\sim$ on $V(\Gamma)$:
\begin{itemize}
\item $v_1\sim v_2$ $\Leftrightarrow$ there exists a path from $v_1$ to $v_2$
in $\Gamma$ with $\mu(p)=\varepsilon$;
\end{itemize}
and results into the graph $\Gamma/\sim$. To effectively represent
the sets of identified vertices
(equivalence classes) one can use \emph{compressed tree representation}
(as in \cite{tarjan84}).
Each vertex $v\in V$ contains a pointer $p(v)$ to its parent,
the root points to itself. Vertex and its parent always belong to the same equivalence class, thus each tree represents an equivalence class.
This presentation allows to
compare and merge two classes very efficiently, which results in
$O(|V| \log^\ast(|V|))$ complexity bound.

\subsubsection{Weighted $X$-digraph folding}

To achieve a similar complexity bound for weighted $X$-digraph folding
we need to take into account shifts of weight. Since in the middle of the folding
process we work with equivalence classes
(as in Section \ref{se:digraph_folding} above), weight-shift
requires shifting weights for a whole class. To avoid shifting weights
many times with each vertex $v$ we keep a number $\delta(v)$ called
\emph{shift value}. That defines the total shift $\Delta(v)$ of a vertex $v$ as:
$$
\Delta(v) =
\begin{cases}
0, &\mbox{ if } p(v)=v;\\
\delta(v)+\Delta(v'), &\mbox{ if } p(v)=v'\ne v.
\end{cases}
$$
Hence, to perform weight-shift of $v_2$ while identifying $v_1$ and $v_2$ (case $v_1 \neq v_2$ in the Folding procedure, section \ref{se:folding})
we set $p(v_2)=v_1$ and instead of doing $Shift(v, c)$ immediately, we just set $\delta(v)=c$. Comparison and merge of two vertex
equivalence classes can be easily extended to tree
presentations with shifts $\delta$.
Therefore, the following proposition holds.

\begin{prop}\label{FoldComplexity}
The number of additions performed by the folding procedure
is $O(|V|\log^\ast(|V|))$.
\qed
\end{prop}

We would like to point out that the values of the weight function can grow
exponentially fast (linearly in binary). Nevertheless, in all our
experiments we never encountered values greater than $2^{64}$.

\begin{cor}\label{co:Rcompletion_complexity}
Let $(X;R)$ be a symmetrized presentation.
The total number of additions required to apply $R$-completion
to a weighted graph $\Gamma$ is $O\left(\left|\Gamma\right| L\left(R\right)\log^\ast\left(\left|\Gamma\right| L\left(R\right)\right)\right)$.
\end{cor}

\subsection{Harvest}\label{se:harvest}

Here we describe a procedure
that for a weighted folded $X$-digraph $\Gamma$ and $L\in\MN$
finds all circuits in $\Gamma$ of weight $1$ and length up to $L$.
Since the number of such circuits is expected to grow exponentially with $L$
one can not expect a very efficient implementation.
Nevertheless, certain heuristics allow to speed up enumeration significantly.

For every vertex $v\in\Gamma$ we find all reduced paths $P$ in $\Gamma$ from $v$
of length up to $\lceil L/2\rceil$ and distribute them into bins $P_{u,x,\nu,l}$:
$$
P_{u,x,\gamma,l} =
\{p\in P\mid t(p)=u,\ \gamma(p)=\gamma,\ |p| = l,\ \mu(p) \mbox{ ends with }x\}.
$$
Then we consider pairs of ``compatible'' bins:
$$
T= \{(P_{u_1,x_1,\gamma_1,l_1},P_{u_2,x_2,\gamma_2,l_2}) \mid u_1=u_2, x_1\ne x_2, \gamma_1-\gamma_2 \equiv_N 1,\ 0\le l_1-l_2\le 1\}.
$$
The set of circuits at $v$ is constructed as:
$$
\{p_1p_2^{-1} \mid p_1\in P_1,\ p_2\in P_2,\ (P_1,P_2)\in T\}.
$$
Finally, the vertex $v$ is removed from $\Gamma$ and the same process is applied to another vertex in $\Gamma$. Note that we may skip vertices which has no
adjacent edges of a non-trivial weight. The number of operations is bounded
by $(2\cdot |X|)^{\lceil L/2\rceil}$, but it was much smaller in practice.

\subsection{ACM-move implementation efficiency}
\label{se:efficacy}

The implementation of the ACM-move described above constructs a subset of the set
of conjugates for a given $u$ in $\gpr{x,y}{v}$ of bounded length.
In general, it can be a proper subset of $U(u,v)$.
The result depends on the value of $D$ --
the number of completion steps used to construct pseudo-conjugacy graphs.
We denote it by $U_D(u,v)$. In this section we shortly prove that:
$$
\bigcup_{D=1}^\infty U_D(u,v) = U(u,v),
$$
and define the parameter $\delta$ (called \emph{depth}) of a conjugate $u'\in U(u,v)$
responsible for ``complexity'' of $u$.

For a set $S \subset \EP$ let $\partial S$ be its boundary
and $\overline{S}$ the closure of $S$ in $\EP$. Let $D$ be a finite connected
planar $X$-digraph with set of vertices $V(D)$ and set of edges $E(D)$.
Let $C(D)$ be a set of \emph{cells} of $D$ which are connected and simply
connected bounded components of $\MR^2\setminus D$. The unbounded component
of $\MR^2\setminus D$ is called the \emph{outer cell} of $D$ denoted by
$c_{\text{out}}$. An edge $e \in E(D)$ is \emph{free} if it does not belong to
$\partial c$ for any $c \in C(D)$. For any $e \in E(D)$
we denote its label by $\mu(e) \in X^{\pm}$. The boundary of a cell $c \in C(D)$
traversed in a counterclockwise direction starting from some vertex of $c$ makes a
closed path $e_1 \dots e_n$ giving the word
$\mu(c) = \mu(e_1) \dots \mu(e_n) \in (X^{\pm})^\ast$ called a \emph{boundary label}
of $c$. Depending on a starting vertex we get a cyclic permutation of the same word.

For the rest of this subsection let $D$ be a finite connected planar $X$-digraph with a \emph{base vertex} $v_0 \in V(D) \cap \partial c_{\text{out}}$.
The graph $D$ is a \emph{van Kampen diagram} over $\gpr{X}{R}$ if $\mu(c) \in R^{\star}$ for every $c \in C(D)$.
The boundary label $\mu(D)$ of $D$ is the boundary label of $\partial c_{\text{out}}$ read starting from $v_0$ in a counterclockwise direction.
Note that we need also to specify the first edge to read from $v_0$, that is, the starting boundary position, but it is not important for our considerations so we omit this issue.

Now let us exclude one of the cells from $C(D)$ and call it the
\emph{inner cell} $c_{\text{in}}$ of $D$. Denote $v_0$ by $v_{\text{out}}$ and
pick any vertex $v_{\text{in}} \in V(D) \cap \partial c_{\text{in}}$. We call $D$
an \emph{annular (Schupp) diagram} (see~\cite{Schupp:1968}) over $\gpr{X}{R}$
if $\mu(c) \in R^{\star}$ for any $c \in C(D)$. Its two boundary
labels $\mu_{\text{in}}(D) = \mu(c_{\text{in}})$ and
$\mu_{\text{out}}(D) = \mu(c_{\text{out}})$ read in a counterclockwise direction
from $v_{\text{in}}$ and $v_{\text{out}}$ correspondingly, are called the \emph{inner}
and \emph{outer} labels of $D$. For any $w_1, w_2 \in (X^\pm)^\ast$ we have that
$w_1 \sim_G w_2$ if and only if there exists an annular diagram $D$ over
$\gpr{X}{R}$ with $\mu_{\text{in}}(D) = w_1$ and $\mu_{\text{out}}(D) = w_2$.

We measure diagram complexity using a notion of depth
(introduced in~\cite{Ushakov:thesis}). For a van Kampen
or annular diagram $D$ define the {\em dual graph} $D^\ast = (V^\ast,E^\ast)$ as
an undirected graph with $V^\ast = C(D) \cup c_{\text{out}}$
(for annular diagrams we add $c_{\text{in}}$) and
$E^\ast = \{(c_1,c_2) \mid \partial c_1 \cap \partial c_2 \ne\emptyset \}$.
We denote the graph distance in $D^\ast$ by $d^\ast$.

The \emph{depth} of a (generalized) van Kampen diagram $D$ is defined by:
\[
    \delta(D) = \max_{c \in C(D)} d^\ast(c, c_{\text{out}}).
\]
The \emph{depth} of an annular diagram $D$ is:
\[
	\delta_\sim(D) = \max_{c \in C(D)} \left[ \min \left( d^\ast(c, c_{\text{out}}), d^\ast(c, c_{\text{in}}) \right) \right].	
\]
(There is a similar notion of a diagram radii
(see~\cite{Gersten-Riley:2002, Rileybook2007}).)
Define the \emph{conjugate depth} of two words $w_1, w_2 \in F(X)$ as:
\[
	\delta_{\sim}(w_1, w_2) =
		\min_{\substack{D \text{ is}\\
						\text{an annular}\\
						\text{diagram}}}
		\Set{\delta(D)}{\mu_{\text{in}}(D) = w_1, \mu_{\text{out}}(D) = w_2}
\]
if $w_1 \sim_{G} w_2$ and $\infty$ otherwise.
The next theorem shows a relation between complexity of the conjugacy search problem
and the conjugacy depth.

\begin{theorem*}[Theorem~3.5 in~\cite{Morar-Ushakov:2015}]
\label{Thm:cspSolverComplexity}
There exists an algorithm which for a given
finite symmetrized presentation $\gpr{X}{R}$
and words $w_1, w_2 \in F(X)$ terminates with the affirmative answer
if and only if $w_1 \sim_{G} w_2$.
Furthermore, its complexity can be bounded above by:
$$
	\tilde{O}\rb{ |w_1| |w_2| L(R)^{2 \delta_{\sim}(w_1, w_2)}}.
$$
\end{theorem*}

For our purposes it will be useful to define another characteristic
of annular diagrams, the \emph{inner conjugacy depth}:
$$
\delta_\sim^{in}(D) = \max_{c \in C(D)} d^\ast(c, c_{\text{in}}),
$$
and the conjugacy depth from $w_1$ to $w_2$ as:
\[
	\delta_{\sim}^{w_1}(w_2) =
		\min_{\substack{D \text{ is}\\
						\text{an annular}\\
						\text{diagram}}}
		\Set{\delta_\sim^{w_1}(D)}{\mu_{\text{in}}(D) = w_1, \mu_{\text{out}}(D) = w_2},
\]
if $w_1\sim_G w_2$ and $\infty$ otherwise.

\begin{thm}\label{th:effectiveness}
Assume that $u$ and $u'$ are conjugate in $\gpr{x,y}{v}$
and $\delta=\delta_\sim^{u}(u')$.
If $\delta\le D$ (where $D$ is the number of $R$-completions
applied to $\Loop(u)$),
then our implementation of the ACM-move
applied for $(u,v)$ produces the pair $(u',v)$.
\end{thm}

\begin{proof}
Same proof as that of \cite[Theorem~17.6.12]{ushakov2011non}.
\end{proof}

It easily follows from Corollary \ref{co:Rcompletion_complexity} that, in general,
$D$ iterations of $R$-completion procedure require exponential time.
Fortunately, in our experiments with $\AK(3)$, we observed that the value $D=2$ is sufficient.
Application of more than two $R$-completions did not produce any additional conjugates
and did not change highlighted figures in Table \ref{tab:presentation_num}, section \ref{se:results:AK3}.

\section{Nielsen automorphisms and AC-equivalence}
\label{se:automorphisms}

In this section we discuss automorphism-moves,
namely applications of an automorphism $\varphi\in\Aut(F_2)$:
$$
(u,v) \to (\varphi(u),\varphi(v)).
$$
It is not known if adding these transformations to AC-moves results
in an equivalent system of transformations or not (even for presentations of $\One$).
Nevertheless, the following is true.

\begin{lem}[{{\cite[Proposition 1(iii)]{Burns-Macedonska:1992} or \cite[Lemma 1]{Miasnikov99geneticalgorithms}}}]
\label{le:old_phi}
If $(u,v)$ can be transformed into $(x,y)$ using AC-moves and
automorphisms, then $(u,v)$ can be transformed into $(x,y)$ using AC-moves only.
\qed
\end{lem}

With any $(u,v)$ we can associate $\varphi_{(u,v)}\in\End(F_2)$ defined by
$\varphi_{(u,v)}(x)=u$ and $\varphi_{(u,v)}(y)=v$. That way we can
treat $\CB_2$ as a monoid.
By Lemma \ref{le:same_phi}, $\CB_2$ naturally acts on $\AC$-components.

\begin{lem}
\label{le:same_phi}
Assume that $(u,v) \sim_{\AC} (u',v')$ and $\varphi\in\CB_2$.
Then:
$$(\varphi(u),\varphi(v))\sim_{\AC}(\varphi(u'),\varphi(v')).$$
\end{lem}

\begin{proof}
Clearly, it is sufficient to prove the result for the case when
$(u',v')$ is obtained from $(u,v)$ by a single ACM-move, i.e.,
we may assume that $u \sim u'$ in $\gpr{x,y}{v}$ and $v=v'$.
Hence, $\varphi(u) \sim \varphi(u')$ in $\gpr{x,y}{\varphi(v)}$
and $(\varphi(u),\varphi(v))\sim_{\AC}(\varphi(u'),\varphi(v'))$.
\end{proof}

Lemma \ref{le:new_phi} immediately implies Lemma \ref{le:old_phi}
since $(x,y)\sim_{AC} (\varphi(x),\varphi(y))$ for every $\varphi\in\Aut(F(X))$.

\begin{lem}
\label{le:new_phi}
Let $u,v\in F(X)$, $\varphi\in\End(F(X))$, and $(u,v) \sim_{\AC} (\varphi(u),\varphi(v))$.
Then for any $u',v'\in F(X)$ the following holds:
$$
(u,v) \sim_{\AC} (u',v')
\ \ \Rightarrow\ \
(u',v') \sim_{\AC} (\varphi(u'),\varphi(v'))
$$
\end{lem}

\begin{proof}
As above, we may assume that $u \sim u'$ in $\gpr{x,y}{v}$ and $v=v'$.
Hence $(\varphi(u'),\varphi(v')) \sim_{AC} (\varphi(u),\varphi(v))\sim_{AC}(u,v)\sim_{AC}(u',v')$.
\end{proof}

With $(u,v) \in \CB_2$ we can associate a monoid:
$$
\End_{\AC}(u,v) = \Set{\varphi\in\End(F_2)}{(u,v)\sim_{\AC} (\varphi(u),\varphi(v))},
$$
under the usual composition. Lemma \ref{le:new_phi} implies that:
\begin{itemize}
\item
$\End_{\AC}(u,v) = \End_{\AC}(u',v')$ whenever $(u,v) \sim_{AC} (u',v')$;
\item
$\Aut(F_2) \le \End_{\AC}(x,y)$.
\item
$\End_{\AC}(x,y) = \Set{\varphi}{\varphi(x)=u,\ \varphi(y)=v,\ (u,v)\in\CB_2}$
if and only if ACC holds.
\end{itemize}
Below we prove that $\Aut(F_2) \le \End_{\AC}(\AK(n))$ for every $n\in\MN$.

\subsection{Akbulut-Kurby examples}

Lemmas \ref{le:AKn_phi_y_x}, \ref{le:AKn_phi_x_Y}, and \ref{le:AKn_phi_x_yx}
show that:
$$\varphi(\AK(n)) \sim_{AC} \AK(n),$$
for some $\varphi\in\Aut(F_2)$.
Proofs were obtained using a computer program.

Below for brevity we use $X$ and $Y$ as $x^{-1}$ and $y^{-1}$.

\begin{lem}
\label{le:AKn_phi_y_x}
$\varphi(\AK(n)) \sim_{AC} \AK(n)$ for $\varphi\in\Aut(F_2)$
defined by $\varphi(x) = y$, $\varphi(y) = x$.
\end{lem}

\begin{proof}
The pair
$(\varphi(xyxYXY),\varphi(x^k Y^{k+1})) = (yxyXYX,y^k X^{k+1})$
can be modified as follows:
\[
\begin{array}{rl}
\sim_{AC_2}& (xyxYXY,y^k X^{k+1})\\
\sim_{ACM}& (xyxYXY,c^{-1} x^k Y^{k+1} c) \mbox{ where } c=xyx\\
\sim_{AC_3}& (xyxYXY,x^{k}Y^{k+1} ).
\end{array}
\]
Appendix \ref{se:used_moves} provides more detail for each ACM-move used.
\end{proof}

\begin{lem}
\label{le:AKn_phi_x_Y}
$\varphi(\AK(n)) \sim_{AC} \AK(n)$ for $\varphi\in\Aut(F_2)$
defined by $\varphi(x) = x$, $\varphi(y) = Y$.
\end{lem}

\begin{proof}
The pair $(\varphi(xyxYXY),\varphi(x^k Y^{k+1})) = (xYxyXy,x^k y^{k+1})$
can be modified as follows:
$$
\begin{array}{rl}
\sim_{AC_2}& (xYxyXy,Y^{k+1} X^k)\\
\sim_{AC_3}& (xyXyxY,Y^{k+1} X^k)\\
\sim_{ACM}& (xyXyxY,c^{-1} xyXy^{k+1}xY^{k+2} c), c=y^{k+1}\\
\sim_{AC_3}& (xyXyxY,xyXy^{k+1}xY^{k+2})\\
\sim_{ACM}& (c^{-1}xyXyxYYXy c,xyXy^{k+1}xY^{k+2}), c=y^{k+2}XyxY\\
\sim_{AC3}& (xyXyxYYXy, xyXy^{k+1}xY^{k+2})\\
\sim_{AC2}& (xyXyxYYXy, y^{k+2}XY^{k+1}xYX)\\
\sim_{ACM}& (xyXyxYYXy,c^{-1} xyXy^kXY^k c), c=xyXy\\
\sim_{AC3}& (xyXyxYYXy, xyXy^kXY^k)\\
\sim_{AC2}& (YxyyXYxYX, xyXy^kXY^k)\\
\sim_{ACM}& (c^{-1} YxyyXYxYX c, xyXy^kXY^k), c=Y^{k-1}xY^{k-1}xYYXy\\
\sim_{AC3}& ( YxyyXYxYX, xyXy^kXY^k)\\
\sim_{ACM}& (c^{-1} x^kY^{k+1} c, xyXy^kXY^k), c=Y^{k-1}xY^{k-1}xYYXy\\
\sim_{AC3}& ( x^kY^{k+1}, xyXy^kXY^k)\\
\sim_{ACM}& ( x^kY^{k+1}, c^{-1} xyxYXY c), c=Y^{k}\\
\sim_{AC3}& ( x^kY^{k+1}, xyxYXY)
\end{array}
$$
\end{proof}

\begin{lem}
\label{le:AKn_phi_x_yx}
$\varphi(\AK(n)) \sim_{AC} \AK(n)$ for $\varphi\in\Aut(F_2)$
defined by $\varphi(x) = x$, $\varphi(y) = yx$.
\end{lem}

\begin{proof}
The pair $(\varphi(xyxYXY),\varphi(x^k Y^{k+1})) = (xyxYXXY, x^{k-1} (YX)^{k} Y)$
can be modified as follows:
	$$
	\begin{array}{rl}
	=& (xyxYXXY, x^{k-1} (YX)^{k} Y)\\
	\sim_{AC_3}& (YxyxYXX,x^{k-1} (YX)^{k} Y)\\
	\sim_{AC_3}& (xxyXYXy,x^{k-1} (YX)^{k} Y)\\
	\sim_{ACM}& (xxyXYXy,c^{-1} x^{k-1}yX^{k-1}yXY c), c=x^{k-3}(YX)^kY\\
	\sim_{AC_3}& (xxyXYXy,x^{k-1}yX^{k-1}yXY)\\
	\sim_{AC_2}& (YxyxYXX,x^{k-1}yX^{k-1}yXY)\\
	\sim_{ACM}& (c^{-1}x^kY^{k+1}c,x^{k-1}yX^{k-1}yXY), c=xyX^{k-1}yXYXy\\
	\sim_{AC_3}& (x^kY^{k+1},x^{k-1}yX^{k-1}yXY)\\
	\sim_{AC_2}& (x^kY^{k+1},yxYx^{k-1}YX^{k-1})\\
	\sim_{ACM}& (x^kY^{k+1},c^{-1} xyxYXY c), c=yX^{k-1}yXY\\
	\sim_{AC_3}& (x^kY^{k+1},xyxYXY)\\
	\end{array}
	$$
\end{proof}

\begin{prop}
$\varphi(\AK(n)) \sim_{AC} \AK(n)$ for any $n > 2$ and $\varphi\in\Aut(F_2)$.
\end{prop}

\begin{proof}
Automorphisms considered in Lemmas \ref{le:AKn_phi_y_x}, \ref{le:AKn_phi_x_Y}
and \ref{le:AKn_phi_x_yx} generate $\Aut(F_2)$.
Hence the proposition holds for any $\varphi\in\Aut(F_2)$.
\end{proof}

The next corollary implies that adding automorphism-moves to AC-moves
does not increase orbits for $\AK(n)$ presentations.

\begin{cor}\label{co:auto}
If $(u, v)\sim_{AC}\AK(3)$, then
$(\varphi(u), \varphi(v))\sim_{AC}\AK(3)$
for every
$\varphi\in\Aut(F_2)$.
\qed
\end{cor}

It is natural to raise a question if a similar result holds for all balanced presentations of $\One$.
We performed experiments with several randomly generated Miller--Schupp presentations.
The results were not always positive, i.e., for some presentations we were unable
to prove $\AC$-equivalence to their automorphic images.

\begin{conj*}
It is not true that for every $(u,v)\in\CB_2$ and $\varphi\in\Aut(F_2)$
$(u,v)\sim_{\AC} (\varphi(u),\varphi(v))$.
\end{conj*}

Note that the conjecture above immediately implies negative answer to ACC.
In Table \ref{tab:MS-presentations} the reader can find particular balanced presentations
suspected to satisfy the conjecture above.

\section{Canonical forms of presentations}\label{se:equivalence_rel}

For a given relation tuple $\ovu\in \CB_n$
the search space $\CC_\ovu$ is infinite
and no computer procedure can exhaust all of its elements.
To reduce the search space
one can introduce an equivalence relation $\sim$
on $\CC_\ovu$ (or on $F_n^n$),
define efficiently computable representatives $NF(\ovu)$
for equivalence classes, and study the quotient space:
$$
\CB_n/\sim \ =\ \{\NF{\ovu} \mid \ovu\in\CB_n\}.
$$
That way one can achieve compression of the search space as
a single element $\NF{\ovu}$ representing its (infinite) equivalence class.
Clearly, coarser relations on $F_n^n$ give better compression.

Below we consider two equivalence relations on $\CB_2$.
Same results hold for $\CB_n$ with $n>2$.
The first one (referred to as a \emph{cyclic relation} here) was used
by Casson in a series of unpublished work (according to Bowman--McCaul)
and by Bowman--McCaul (who ``followed'' Casson).
The second relation is new and is significantly stronger.

\subsection{Cyclic relation}
\label{se:cyclic_relation}

Let $\sim$ be the transitive closure of the following pairs in $F_2^2$:
\begin{itemize}
\item
$(u,v) \sim (v,u)$,
\item
$(u,v) \sim (u^{-1},v)$,
\item
$(u,v) \sim (u,c^{-1}vc)$,
\end{itemize}
where $u,v,c$ are arbitrary words in $F_2$.
We call $\sim$ a \emph{cyclic relation} on $F_2^2$.

To define canonical representatives for the
cyclic relation we do the following.
Fix any order on generators, say $x_1<x_1^{-1}<x_2<x_2^{-1}$
and denote by $<$ the corresponding shortlex order on $F_2$
and, further, the corresponding lexicographic order on $F_2^2$.
Let $(u,v) \in F_2^2$. It is easy to see that taking
the least cyclic permutation of $u^{\pm 1}$,
the least cyclic permutation of $v^{\pm 1}$,
and ``sorting'' the obtained words,
produces the least representative of the equivalence class of $(u,v)$,
denoted by $\NF{u,v}$. Clearly, so-defined normal form is an efficiently computable.

It easily follows from the definition that
$\NF{u,v}\in\CB_2$ for any $(u,v)\in\CB_2$.
Hence, AC-moves can be naturally defined on $\CB_2/\sim$:
$$
\NF{u,v} = (u,v) \ \ \stackrel{\nu}{\mapsto} \ \ \NF{\nu(u,v)},
$$
where $\nu$ is an $\AC$-move.
The problem with this approach is that computing the cyclic normal form
negates applications of the $\AC_3$-move.
That can result in the component $\CC_\ovu$ being broken into
subcomponents (i.e., $\CC_\ovu$ can become disconnected).
In particular, Bowman--McCaul implementation
(found here \url{http://www.math.utexas.edu/users/sbowman/ac-bfs
.tar.gz}) does not take the normal form of a pair
obtained by an $\AC_3$-move,
which completely negates the advantage of using normal forms.
As we explain below, ACM-moves can solve this problem.

\subsection{Cyclic relation with automorphisms}
\label{se:equivalence}

Define an equivalence relation $\sim$ on $F_2^2$ by taking a closure
of the following pairs:
\begin{itemize}
\item
$(u,v) \sim (v,u)$,
\item
$(u,v) \sim (u^{-1},v)$,
\item
$(u,v) \sim (u,c^{-1}vc)$,
\item
$(u,v) \sim (\varphi(u), \varphi(v))$,
\end{itemize}
where $u,v,c$ are arbitrary words in $F_2$ and
$\varphi$ is an arbitrary automorphism in $\Aut(F_2)$.
Note that the defined relation makes $(u,v)$ equivalent to $(\varphi(u),\varphi(v))$
which, in general, is not known to be $\AC$-equivalent to $(u,v)$.
Hence, it is possible that an equivalence class of
$(u,v)$ contains an element which is not $\AC$-equivalent to $(u,v)$.
Nevertheless, the following is true.

\begin{prop}\label{pr:AK-n-Aut_moves}
For every $u,v,u',v'\in F_2$ if
\begin{itemize}
\item
$\AK(n)\sim_{AC}(u,v)$, and
\item
$(u,v) \sim (u',v')$,
\end{itemize}
then $\AK(n)\sim_{AC}(u',v')$.\end{prop}

\begin{proof}
Follows from Corollary \ref{co:auto}.
\end{proof}

Proposition \ref{pr:AK-n-Aut_moves} allows us to replace
the original component $\CC=\CC_{\AK(3)}$ with its factor $\CC/\sim$
which is much smaller.
The problem that taking a normal form of a
pair negates $\AC_3$-moves
is still relevant if we use the original $\AC$-moves.
That is where ACM-moves really help.
It follows from Theorem \ref{th:effectiveness}
that choosing  sufficiently large value of the parameter
$D$ we can produce any conjugate of $u$ in $\gpr{x,y}{v}$.

As in Section \ref{se:cyclic_relation},
the \emph{normal form} of the pair $(u,v)$ is defined as the minimal
pair in its equivalence class. Below we show that normal forms can be computed
efficiently. Our main tool is the following classic result.

\begin{theorem*}[Whitehead theorem, see {\cite[Proposition 4.20]{lyndon2015combinatorial}}]
\label{thm:whitehead}
Let $w_1, \ldots, w_t$, $w_1', \ldots, w_t'$ be cyclic words in a free group $F$
such that:
$$
w_1' = \alpha(w_1), \ldots, w_t' = \alpha(w_t)
$$
for some $\alpha\in\Aut(F)$.
If $\sum |w_k'|$ is minimal among
all $\sum |\alpha'(w_k)|$ for $\alpha'\in\Aut(F)$, then
$\alpha = \tau_1\ldots\tau_n$, $n\geq 0$, where
  $\tau_1,\ldots,\tau_n$ are \emph{Whitehead} automorphisms
  and for each $i$ the length
  $\sum_k |\tau_1\cdots \tau_i(w_k)| \leq \sum_k |w_k|$ with strict
  inequality unless $\sum|w_k| = \sum|w_k'|$.
  \qed
\end{theorem*}

Recall \cite[Section 1.4]{lyndon2015combinatorial} that
Whitehead automorphisms are automorphisms of two types:
\begin{itemize}
\item
(length-preserving) automorphisms that permute the letters $X^\pm$;
\item
automorphisms that for some fixed ``multiplier'' $a\in X^\pm$
carry each of the elements $x\in X$ into one of $x$, $xa$, $a^{-1}x$, or $a^{-1}xa$.
\end{itemize}
There are exactly $20$ Whitehead automorphisms for a free group of rank $2$.
According to the Whitehead theorem
if the total length of a given
tuple of cyclic words can be decreased by an application of an automorphism,
then it can be decreased by an application of a single Whitehead automorphism.
Hence, to compute the normal form of a pair $(u,v)$ we do the following.
\begin{itemize}
\item
First, minimize the total length $|u|+|v|$ of $(u,v)$
by applying $12$ non-length-preserving Whitehead automorphisms
while the total length decreases.
\item
Then, construct a set of all equivalent pairs of the least total length
by applying all automorphisms.
\item
Finally, choose the least cyclic normal form among the pairs of the least length.
\end{itemize}
The procedure described above is efficient except, maybe, the second step,
where we construct the set of all pairs of the least total length.
Currently there are no theoretical polynomial bounds on the size of that set.
Nevertheless, in our computations the maximal size observed
was $112$ for $\AK(3)$ equivalent presentations with $|r| \le 20$ bound. The average size
of the set of all pairs of the least total length was $9$.

\section{Groups with high Dehn function}
\label{se:org_groups}

One potential challenge for computer enumeration techniques is described by
Bridson in \cite{bridson2015complexity}
and Lishak in \cite{lishak2015balanced}.
Both papers use a similar idea based on properties of the following one-relator group:
\begin{equation}\label{eq:Baumslag}
\gpr{x,y}{y^{-1}x^{-1}yx y^{-1}xy = x^2},
\end{equation}
introduced by Baumslag in \cite{Baumslag:1969}, satisfying the inequality:
\begin{equation}\label{eq:Baumslag-Dehn}
Dehn(n) \ge Tower_2(\log_2(n)),
\end{equation}
first observed in \cite{Gersten:1992}.
Lishak constructs a particular sequence of
balanced presentations parametrized by $n\in\MN$:
$$\ovu_n = (r,w_ny^{-1}) \in\CB_2,$$
where $w_n\in F(X)$, satisfying the following conditions.
\begin{itemize}
\item
$\ovu_n$ is $\AC$-equivalent to the canonical presentation.
\item
The number of steps required to obtain the canonical presentation
is super-exponential in $n$.
\end{itemize}
The later property comes as a consequence of the inequality (\ref{eq:Baumslag-Dehn}).

Being very curious about possibility that that is the reason why
our program fails to find $\AC$-trivialization of $\AK(3)$ we tested
all words obtained in our experiments. For each word $r$ we attempted
to bound the Dehn function of the group $\gpr{x,y}{r}$.
For that purpose we used D. Holt's package \cite{holt2000} to identify automatic groups
(automatic groups have at most quadratic Dehn functions) that left us
with $5356$ ``perhaps non-automatic'' one-relator groups.
Among those $1205$ we classified as Baumslag--Solitar type presentations, i.e.,
the presentations with the relation $(u^n)^v = u^m$ for some $u,v\in F_2$. None
of Baumslag--Solitar type presentations satisfied the condition $u\sim_{F} v$,
i.e., no presentations were identified as Baumslag-type presentations \eqref{eq:Baumslag}.
Clearly, this is a heuristic approach and we can not guarantee that our list
of presentations does not contain Baumslag groups, as isomorphism problem
for one-relator groups is not known to be decidable/undecidable.
Also, we were unable to classify $4151$ remaining presentations. In case someone would like to further investigate this, we have published the obtained lists here: \url{https://github.com/stevens-crag/ak3_types}.

In light of these heuristic results it seems to be a very interesting computational
problem to classify short one-relator groups $\gpr{x,y}{r}$ with $|r|\le 20$.
Find precise upper bounds on their Dehn functions.

\begin{conj*}
Baumslag's group $\gpr{x,y}{y^{-1}x^{-1}y x y^{-1}xy = x^2}$
has the highest Dehn function among all one-relator groups.
\end{conj*}


\section{Results}
\label{se:results}

The described algorithms were tested on several known potential counterexamples.
Our attention was mainly focused on $\AK(3)$ and Miller--Schupp presentations.
To test performance and compare with other experimental results
we also ran our programs on $\AK(2)$
and other presentations that are known to be $\AC$-equivalent to the
canonical presentation.

As we already mentioned in Section \ref{se:harvest},
we set a bound $L$ on the length on the conjugates
obtained during harvest phase.
We also set a limit on the total length of pairs to be $2L + 2$.
Notice that we need to do that as taking a normal form
described in Section \ref{se:equivalence} can increase length of one of the words
beyond $L$ (which is allowed in our implementation).
Experiments were run on a machine
with two $8$-core $3.1$ Ghz Intel Xeon CPU E5-2687W and $64$ GB RAM.

\subsection{Enumeration of $\AK(3)$-equivalent presentations}
\label{se:results:AK3}

As shown in Section \ref{se:automorphisms}, automorphism-moves can be used together
with ACM-moves when applied to $\AK(3)$-equivalent presentations.
In particular, one can use normal forms from Section \ref{se:equivalence}
to compress the array of stored presentations. Table \ref{tab:presentation_num}
shows dynamics of growth of a component of $\CC_{\AK(3)}$ constructed
by our program for different values of $L$. Each cell in Table \ref{tab:presentation_num}
corresponds to a value $L$ and a value $T$ and presents the number of pairs
of the total length equal to $T$ constructed by the program with
the single-word-bound $L$.

\begin{table}[h]
\centering
{\tiny
\begin{tabular}{ l | c | c | c | c | c | c | c | c | c | c | c }
T$\backslash$L & 10 & 11 & 12 & 13 & 14 & 15 & 16 & 17 & 18 & 19 & 20 \\ \hline
13 &  \textbf{4} &   \textbf{4} & \textbf{4} & \textbf{4} & \textbf{4} & \textbf{4} & \textbf{4} & \textbf{4} & \textbf{4} & \textbf{4} & \textbf{4}\\ \hline
14 &  \textbf{10} &  \textbf{10} & \textbf{10} & \textbf{10} & \textbf{10} & \textbf{10} & \textbf{10} & \textbf{10} & \textbf{10} & \textbf{10} & \textbf{10} \\ \hline
15 &  \textbf{70} &  \textbf{70} & \textbf{70} & \textbf{70} & \textbf{70} & \textbf{70} & \textbf{70} & \textbf{70} & \textbf{70} & \textbf{70} & \textbf{70}\\ \hline
16 &  64 &  \textbf{86} & \textbf{86} & \textbf{86} & \textbf{86} & \textbf{86} & \textbf{86} & \textbf{86} & \textbf{86} & \textbf{86} & \textbf{86}\\ \hline
17 & 220 & 416 & 454 & \textbf{458} & \textbf{458} & \textbf{458} & \textbf{458} & \textbf{458} & \textbf{458} & \textbf{458} & \textbf{458} \\ \hline
18 &  98 & 392 & 398 & \textbf{590} & \textbf{590} & \textbf{590} & \textbf{590} & \textbf{590} & \textbf{590} & \textbf{590} & \textbf{590}\\ \hline
19 & 240 & 764 & 1382 & 2854 & \textbf{3226} & \textbf{3226} & \textbf{3226} & \textbf{3226} & \textbf{3226} & \textbf{3226} & \textbf{3226}  \\ \hline
20 &  10 & 442 & 522 & 2004 & 2082 & 3352 & 3352 & \textbf{3356} & \textbf{3356} & \textbf{3356} & \textbf{3356} \\ \hline
21 &  20 & 746 & 1624 & 3870 & 8334 & 16948 & 19666 & 19690 & 19690 & \textbf{19692} & \textbf{19692} \\ \hline
22 &   0 & 438 & 570 & 2812 & 3714 & 12288 & 12584 & 23174 & 23174 & 23188 & 23192 \\ \hline
23 &   0 & 112 & 1462 & 4474 & 9194 & 21678 & 41492 & 101544 & 128356 & 128380 & 128388 \\ \hline
24 &   0 &   6 & 42 & 3400 & 3858 & 12978 & 15458 & 61100 & 64686 & 150264 & 150276 \\ \hline
25 &   0 &   0 & 110 & 4350 & 11246 & 22422 & 42550 & 102262 & 236860 & 631000 & 843778 \\ \hline
26 &   0 &   0 & 0 & 4306 & 5384 & 17930 & 19668 & 62874 & 83902 & 375818 & 394172 \\ \hline
27 &   0 &   0 & 0 & 710 & 13548 & 28176 & 51590 & 96714 & 196098 & 538380 & 1269016 \\ \hline
28 &   0 &   0 & 0 & 52 & 494 & 26008 & 27874 & 76930 & 83864 & 289920 & 364040\\ \hline
29 &   0 &   0 & 0 & 0 & 1652 & 30934 & 77162 & 123178 & 230774 & 445036 & 953378 \\ \hline
30 &   0 &   0 & 0 & 0 & 2 & 20430 & 24146 & 128556 & 138478 & 355754 & 405746\\ \hline
31 &   0 &   0 & 0 & 0 & 0 & 5854 & 62178 & 159086 & 368336 & 546680 & 1041462 \\ \hline
32 &   0 &   0 & 0 & 0 & 0 & 326 & 3338 & 122164 & 130302 & 597064 & 639362 \\ \hline
33 &   0 &   0 & 0 & 0 & 0 & 0 & 6314 & 151550 & 353810 & 730650 & 1758270 \\ \hline
34 &   0 &   0 & 0 & 0 & 0 & 0 & 62 & 128556 & 150518 & 538278 & 585132\\ \hline
35 &   0 &   0 & 0 & 0 & 0 & 0 & 0 & 22772 & 374246 & 872784 & 1519374\\ \hline
36 &   0 &   0 & 0 & 0 & 0 & 0 & 0 & 1848 & 19030 & 762768 & 813708 \\ \hline
37 &   0 &   0 & 0 & 0 & 0 & 0 & 0 & 0 & 51496 & 1016332 & 2112918\\ \hline
38 &   0 &   0 & 0 & 0 & 0 & 0 & 0 & 0 & 522 & 848998 & 946260\\ \hline
39 &   0 &   0 & 0 & 0 & 0 & 0 & 0 & 0 & 0   & 209668 & 2414958\\ \hline
40 &   0 &   0 & 0 & 0 & 0 & 0 & 0 & 0 & 0   & 19332 & 120852\\ \hline
41 &   0 &   0 & 0 & 0 & 0 & 0 & 0 & 0 & 0   & 0 & 270942\\ \hline
42 &   0 &   0 & 0 & 0 & 0 & 0 & 0 & 0 & 0   & 0 & 12062\\ \hline
\end{tabular}
}
\caption{Each cell shows the number of pairs $\AC$-equivalent to $\AK(3)$ of total length $T$
obtained by the program when run with the length bound $L$. Highlighted cells do not increase when $L$ is increased.}
\label{tab:presentation_num}
\end{table}

It took our program $10$ days to finish enumeration with the bound $L=20$,
consuming $207$ days of CPU time.
The running time with the bound $L=21$ is expected to be $60$ days.
We decided not to proceed beyond the value $L=20$.
Memory usage during the experiments was moderate and never exceeded $8$Gb.
CPU time is the main obstacle.
However, we can notice that the numbers in rows of Table \ref{tab:presentation_num} stabilize,
at least for values $T=13,\ldots,20$. For instance, we can conjecture that the number
of normal forms of $\AK(3)$-equivalent presentations of total length $20$ or less is $3356$
and there is no canonical one among them.

\subsection{Old non-counterexamples}

We also tested our program on some balanced presentations
that were eliminated from the list of potential counterexamples before us.
Our program trivializes any of them almost immediately (in less than $10$ seconds
on a single computational core)
in less than $5$ ACM-moves.
\begin{itemize}
\item $\AK(2)$: $(x^2y^{-3}, xyx(yxy)^{-1}) \sim_{AC} (x, xyx^{-1}yxy^{-3}) \sim_{AC} (x, y)$
\item Gordon presentation $(x^{-1}yx^2y^{-1}:xy^3x^{-1}y^{-4}) \sim_{AC} (xyx^{-1}y^{-2}, x) \sim_{AC} (x, y)$,
also considered in \cite{bowman2006}.
\end{itemize}

\subsection{Miller--Schupp type presentations}

We analyzed several randomly generated Miller--Schupp presentations:
$$
\gpr{x,y}{x^{-1} y^2 x = y^3,\ x=w},
$$
where $w$  has exponent sum $0$ on $x$. We attempted to trivialize them
or show $\AC$-equivalence with their automorphic images.
Both tasks were dealt with different success.
Table \ref{tab:MS-presentations}
contains pairs $(u,v)$ for  which the program failed to prove equivalence
$(u,v)\sim_{\AC} (\varphi(u),\varphi(v))$ for $\varphi\in\Aut(F_2)$
defined by $y\to y^{-1}$, $y\to yx$ and $x\to y, y\to x$.
(In particular, we could not trivialize the corresponding presentations.)
Table \ref{tab:MS-presentations2} contains Miller-Schupp
type presentations for which the program proved automorphic equivalence
$(u,v)\sim_{\AC} (\varphi(u),\varphi(v))$ for any $\varphi\in\Aut(F_2)$, but
was not able to trivialize them.
Table \ref{tab:MS-presentations3} contains trivializable Miller-Schupp presentations.
The purpose of Tables \ref{tab:MS-presentations}, \ref{tab:MS-presentations2},
and \ref{tab:MS-presentations3} is to provide reference for future experiments.

\begin{table}[h]
\centering
{\tiny
\begin{tabular}{| l | l | l |}
\hline
(xyyyXYY,xxxyXYXY) & (xyyyXYY,xxxYXyXy) & (xyyyXYY,xxxYxyXXXy)\\
\hline
(xyyyXYY,xxxYXXyXXY) & (xyyyXYY,xxxxYXyXXy) & (xyyyXYY,xxxxyXXYXy)\\
\hline
(xyyyXYY,xxxyxyXXXy) & (xyyyXYY,xxxxYXYXXY) & (xyyyXYY,xxxxyXYXXY)\\
\hline
(xyyyXYY,xxxxyXyXXY) & (xyyyXYY,xxxxyXXyXy) & (xyyyXYY,xxxyXyXXXY)\\
\hline
(xyyyXYY,xxxyXyXXXy) & (xyyyXYY,xxxxYXXYXY) & (xyyyXYY,xxxxyXXYXY)\\
\hline
(xyyyXYY,xxxYxyXXXY) & (xyyyXYY,xxxxYXXyXy) & (xyyyXYY,xxxxYXXYXy)\\
\hline
(xyyyXYY,xxxxYXyXXY) & (xyyyXYY,xxxyxYXXXy) & (xyyyXYY,xxxxYXXyXY)\\
\hline
(xyyyXYY,xxxxyXXyXY) & (xyyyXYY,xxxxYXYXXy) & (xyyyXYY,xxxyxyXXXY)\\
\hline
(xyyyXYY,xxxyxYXXXY) & (xyyyXYY,xxxxyXXyXy) & (xyyyXYY,xxxxyXYXXy)\\
\hline
(xyyyXYY,xxxyXXYXXy) & (xyyyXYY,xxxxyXyXXy) & \\
\hline
\end{tabular}
}
\caption{Miller-Schupp pairs $(u,v)$ with unknown equivalence
$(u,v)\sim_{\AC} (\varphi(u),\varphi(v))$ for $\varphi\in\Aut(F_2)$.}
\label{tab:MS-presentations}
\end{table}

\begin{table}[h]
\centering
{\tiny
\begin{tabular}{|l|l|l|}
\hline
(xxxyXXY,xyyyyXYYY) & (xxxyXXY,xyyyXYYYY) & (xyyyXYY,xxxyyXXY)\\
\hline
(xxxyXXY,xxyyyXYY) & (xxxYXXy,xxYYYXyy) & (xyyyXYY,xxxyXyXY)\\
\hline
(xxxYXXy,xyyyxYYYY) & (xyyyXYY,xxxxYXXXY) & (xxxyXXY,xyyyyxYYY)\\
\hline
(xxxYXXy,xxYYxyyy) & (xyyyXYY,xxxyyXXy) & (xyyyXYY,xxxxyXXXy)\\
\hline
(xyyyXYY,xxxyXyXy) & (xxxyXXY,xxyyxYYY) & (xyyyXYY,xxxYXyXY)\\
\hline
(xyyyXYY,xxxyXYXy)\\
\hline
\end{tabular}
}
\caption{Miller-Schupp pairs $(u,v)$ with equivalence
$(u,v)\sim_{\AC} (\varphi(u),\varphi(v))$ for $\varphi\in\Aut(F_2)$,
but not known if trivializable.}
\label{tab:MS-presentations2}
\end{table}

\begin{table}[h]
\centering
{\tiny
\begin{tabular}{|l|l|l|}
\hline
(XyyxYYY,xxYYYXYxYXYY) & (XyyxYYY,xxyyyXYYXyxY) & (XyyxYYY,xxYXyxyyyXY) \\
(XyyxYYY,xxYXyXyyxyy) & (XyyxYYY,xxyyyXYxYXYY) & (XyyxYYY,xyxYYYYYYXy) \\
(XyyxYYY,xyxYXyy) & (XyyxYYY,xxYYXYXYXyy) & (XyyxYYY,xyyxYYYXYYY) \\
(XyyxYYY,xyxYYXyyyyyy) & (XyyxYYY,xyXyyxYXyyxY) & (XyyxYYY,xyyyXyyXyy) \\
\hline
\end{tabular}
}
\caption{Trivializable Miller-Schupp presentations.}
\label{tab:MS-presentations3}
\end{table}

\section{Conclusion}
\label{se:conclusion}

Despite a lot of effort, we were unable to disprove any new
Akbulut--Kurby type presentations.
In fact, the numbers in Table \ref{tab:presentation_num} rows
stabilize as the value of the parameter $L$ increases, suggesting that
the $\AC$-equivalence class of $\AK(3)$ does not contain the canonical presentation,
thus supporting a common opinion that ACC does not hold.


\appendix
\section{Used ACM-moves justification}
\label{se:used_moves}
In this section we prove the one-relator groups identities used in lemmas \ref{le:AKn_phi_y_x}, \ref{le:AKn_phi_x_Y} and \ref{le:AKn_phi_x_yx} for the ACM moves. Every proof demonstrates that $c^{-1} u c (u')^{-1} = 1$ in $\langle x, y \mid v\rangle$.

\subsection{Used in lemma \ref{le:AKn_phi_y_x}} \

\noindent $c = xyx$, $u = x^k Y^{k+1}$,
$v = xyxYXY$, $u' = y^k X^{K+1}$:
\begin{align*}
& XYX\cdot x^k (Y^{k+1} \cdot xyx \cdot x^{k+1}) Y^k\\
=&XYX(x^k \cdot xy x \cdot Y^K)   \quad  (Yxyx = xy)\\
=&XYXxyx \quad (xyxY = yx)\\
=&1
\end{align*}

\subsection{Used in lemma \ref{le:AKn_phi_x_Y}} \

\noindent $u = xyXy^{k+1}xY^{k+2}$, $v = xyXyxY$,
$c = y^{k+1}$, $u' = Y^{k+1}X^k$:
\begin{align*}
&(Y^{k+1}) \cdot xyXy^{k+1}x(Y^{k+2} \cdot y^{k+1}) \cdot x^k (y^{k+1})
=(x^{k+1}yXy^{k+1})x\cdot Y\\
=&(x^k yX y^k)\cdot xY
=\ldots
=yX xY
=1
\end{align*}

\noindent $u = xyXyxYYXy$, $v = xyXy^{k+1}xY^{k+2}$,
$c = y^{k+2}XyxY$, $u' = xyXyxY$:
\begin{align*}
&yXY  xY^{k+2} \cdot xyXyxYYX(y \cdot y^{k+2}XyxY \cdot yXYxY)X\\
=& yXY (xY^{k+2} xyXy)xYYX \cdot y^{k+2} X
= yXY \cdot Y^k\cdot xY(YX y^{k+2} X)\\
=& (yX) Y^{k+1}(xY\cdot Xy^{k+1})
= Y^{k+1}\cdot y^{k+1}
= 1
\end{align*}

\noindent $u = xyXy^kXY^k$, $v = xyXyxYYXy$,
$c = xyXy$, $u' = y^{k+2}XY^{k+1}xYX$:
\begin{align*}
&(Y xYX) \cdot (xyX y^k)XY^k \cdot xyXy \cdot xyXy^{k+1}x(Y^{k+2})\\
=& XY^k xy(Xy x yXy^{k+1})x\cdot Y^3
= X(Y^k xy\cdot y)yX y^{k}\cdot x Y^3\\
=& X\cdot Y^{k-1}xy(Xyx\cdot yX y^{k}) x Y^3
= X (Y^{k-1}xy\cdot y)yX y^{k-1}\cdot x Y^3\\
=& X xy\cdot y yX\cdot x Y^3 = 1
\end{align*}

\noindent $u = x^kY^{k+1}$, $v = xyXy^kXY^k$,
$c = Y^{k-1}xY^{k-1}xYYXy$, \\
$u' = YxyyXYxYX$:
\begin{align*}
(Y&xy yX)y^{k-1}Xy^{k-1}\cdot x^k Y^{k+1} \cdot Y^{k-1}xY^{k-1}xYYXy \\
\cdot& xyXy(xYYXy)
= (y^{k-1}X)y^{k-1} x^k Y^{2k} xY^{k-1} xYYXy (xyXy) \\
=& (y^k\cdot y^{k-1}) x^k Y^{2k} xY^{k-1} xYYX(y)  \quad  (xyXy^kX = y^k)\\
=& y^k x (X y^{k} \cdot x^k) Y^{2k} xY^{k-1} xYYX  \quad  (\text{shift})\\
=& y^k x \cdot y(Xy^k \cdot x^{k-1}) Y^{2k} xY^{k-1} xYYX  \quad  (Xy^kx = yXy^k)\\
& \ldots\\
=& y^k x \cdot y^k X (y^k \cdot Y^{2k}) xY^{k-1} xYYX\\
=& y^k x (y^k X\cdot Y^{k}\cdot x) Y^{k-1} xYYX
= (y^k x) \cdot xY \cdot Y^{k-1} xY(YX)\\
=& (Xy^k\cdot x) Y^{k} xY
= (Xy^k\cdot x) Y^{k} (xY)
= y^k\cdot Y^{k}
= 1
\end{align*}

\noindent $u = xyxYXY$, $v = x^kY^{k+1}$,
$c = Y^{k}$, $u' = y^k x Y^k x YX$:
\begin{align*}
(y^k\cdot xy)xYXY \cdot Y^k \cdot xyX(y^kXY^k) = xYX(Y Y^k) xyX\cdot y^{k+1} = \\
xY(X\cdot X^k x)yX y^{k+1} = x(Y\cdot Y^{k+1}\cdot y)X y^{k+1} = \ldots = Y^{k+1} y^{k+1}
\end{align*}

\subsection{Used in lemma \ref{le:AKn_phi_x_yx}} \

\noindent $u = x^{k-1}yX^{k-1}yXY$, $v = xxyXYXy$,
$c = x^{k-3}(YX)^kY$, \\
$u'= x^{k-1} (YX)^{k} Y$:
\begin{align*}
& y(xy)^k(X^{k-3} \cdot x^{k-1})yX^{k-1}yXY \cdot
(x^{k-3}(YX)^kY \cdot  y (xy)^k X^{k-1}) \\
=& y(xy)^k\cdot x^2 yX^{k-1}(yXY\cdot X)X
= (yx)^ky x^2 yX^{k-1}\cdot XX(Y \cdot X) \\
=& \ldots = y x^2 yX^{k-1} XX\cdot x^kYX
= y x^2 yXYX
= 1
\end{align*}

\noindent $u = x^kY^{k+1}$, $v = x^{k-1}yX^{k-1}yXY$,
$c = xyX^{k-1}yXYXy$, \\
$u'= YxyxYXX$:
\begin{align*}
&(Yx yxY)x^{k-1}Y(X\cdot x^k)Y^{k+1} \cdot xyX^{k-1}yXYXy \cdot xx(yXYXy) \\
=& (x^{k-1}Yx^{k-1}Y^{k+1}) \cdot xyX^{k-1}yXYXy (xx) = \ldots\\
=& Y(x^{k-1} \cdot xyX^{k-1}yXY)Xy
= Y\cdot x\cdot Xy
= 1
\end{align*}

\noindent $u = xyxYXY$, $v = x^kY^{k+1}$,
$c = yX^{k-1}yXY$, \\
$u' = yxYx^{k-1}YX^{k-1}$:
\begin{align*}
&(yx Yx^{k-1}Y)\cdot xyxYX(Y \cdot y)X^{k-1}yXY \cdot x^{k-1}(yX^{k-1}yXY) \\
=& (x)yxY X^k yX(Y x^{k-1})
= (y)xY X^k yX\cdot (y^k)\\
=& (xY) X^k y(X \cdot x^k)
= X^k y \cdot y^k
= 1
\end{align*}

\bibliographystyle{plain}

\end{document}